\documentclass[11pt,a4paper]{amsart}

\usepackage{times,epsfig}
\usepackage{amsfonts,amscd,amsmath,amssymb,amsthm}

\usepackage{bbm}
\numberwithin{equation}{section}

\DeclareSymbolFont{SY}{U}{psy}{m}{n}
\DeclareMathSymbol{\emptyset}{\mathord}{SY}{'306}

\DeclareMathSymbol{\newtimes}{\mathbin}{SY}{'264}

\newcommand{\R}{\mathbb{R}}
\newcommand{\T}{\mathbb{T}}
\newcommand{\C}{\mathbb{C}}
\newcommand{\Z}{\mathbb{Z}}
\newcommand{\N}{\mathbb{N}}

\newcommand{\cA}{{\mathcal A}}

\newcommand{\cF}{{\mathcal F}}
         
\newcommand{\cH}{{\mathcal H}}

\newcommand{\cJ}{{\mathcal J}}
\newcommand{\cK}{{\mathcal K}}
\newcommand{\cL}{{\mathcal L}}

\newcommand{\cN}{{\mathcal N}}

\newcommand{\cP}{{\mathcal P}}
\newcommand{\cQ}{{\mathcal Q}}
\newcommand{\cR}{{\mathcal R}}

\newcommand{\cT}{{\mathcal T}}

\newcommand{\e}{\mathrm{e}}

\newcommand{\1}{\mathbbm 1}

{\bf}{\it}
{\bf}{\it}
{\bf}{\it}
{\bf}{\it}
{\bf}{\it}

\newtheorem{theorem}{Theorem}[section]{\bf}{\it}
\newtheorem{proposition}[theorem]{Proposition}{\bf}{\it}
{\bf}{\it}
\newtheorem{examples}[theorem]{Examples}{\it}{\rm}
{\it}{\rm}
\newtheorem{lemma}[theorem]{Lemma}{\bf}{\it}
\newtheorem{remark}[theorem]{Remark}{\it}{\rm}
{\it}{\rm}
\newtheorem{definition}[theorem]{Definition}{\bf}{\it}
{\bf}{\it}
{\bf}{\it}
{\bf}{\it}

\DeclareMathAlphabet{\Ma}{U}{msa}{m}{n}
\DeclareMathAlphabet{\Mb}{U}{msb}{m}{n}

\DeclareMathAlphabet{\Meuf}{U}{euf}{m}{n}
\DeclareSymbolFont{ASMa}{U}{msa}{m}{n}
\DeclareSymbolFont{ASMb}{U}{msb}{m}{n}
\DeclareMathSymbol{\hrist}{\mathord}{ASMa}{"16}
\DeclareMathSymbol{\varkappa}{\mathalpha}{ASMb}{"7B}
\DeclareMathSymbol{\CrPr}{\mathord}{ASMb}{"6F}

\def\got#1{\Meuf{#1}}
\def\ot #1.{{\got{#1}}}

\newcommand{\al}[1]{\mathcal{#1}}

\title{On spectral approximation, F\o lner sequences and crossed products}

\author[F.~Lled\'{o}]{Fernando Lled\'{o}}
\address{Department of Mathematics,
University Carlos~III, Madrid, Avda. de la Universidad 30, E-28911 Legan\'es
(Madrid), Spain and ICMAT, c. Nicol\'as Cabrera, 13-15
Campus Cantoblanco, UAM, 28049 Madrid}
\email{flledo@math.uc3m.es}

\subjclass[2010]{47L65, 46L60, 43A07} 
\keywords{spectral approximation, C*-algebras,
crossed products, F\o lner sequences, quasidiagonality, amenable groups, rotation algebra}

\date{\today}

\dedicatory{Dedicated to Paco Marcell\'an on his 60th birthday}

\thanks{This work is partly supported by the project MTM2009-12740-C03-01
of the Spanish {\em Ministerio de Ciencia e Innovaci\'on}. Parts of this 
article have been written, while the author was visiting 
CRM in Barcelona.}

\begin{document}

\begin{abstract}
In this article we study F\o lner sequences for operators and mention their relation
to spectral approximation problems.
We construct a canonical F\o lner sequence for the crossed product of a discrete amenable group $\Gamma$
with a concrete C*-algebra $\cA$ with a F\o lner sequence. We also 
state a compatibility condition for the action of $\Gamma$ on $\cA$.
We illustrate our results with
two examples: the rotation algebra (which contains interesting operators like almost Mathieu operators or
periodic magnetic Schr\"odinger operators on graphs) 
and the C*-algebra generated by bounded Jacobi operators. These examples can be interpreted
in the context of crossed products. The crossed products considered can be also seen as a more
general frame that included the set of generalized band-dominated operators.
\end{abstract}

\maketitle
\tableofcontents

\section{Introduction}\label{sec:intro}

Given a sequence of linear operators
$\{T_n\}_{n\in\N}$ in a complex separable Hilbert space $\cH$ that approximates
an operator $T$ in a suitable sense, a natural question is how do the spectral 
objects of $T$ (the spectrum, spectral measures, numerical ranges, pseudospectra etc.)
relate with those of $T_n$ as $n$ grows.
(Excellent books that include a large number of examples and references 
are, e.g., \cite{bChatelin83,bAhues01}.)
In such generality this problem is almost impossible to address. In fact, one
can easily produce examples with bad spectral approximation behavior,
where the spectrum suddenly expands or contracts in the limit
(see, e.g., pp.~289-291 in \cite{bReed80}). Also spectral pollution effects
or spurious eigenvalues can appear as a consequence of the approximation 
process (cf.~\cite{Davies04}).
One of the standard methods to treat these problems
is to compress $T$ to a finite dimensional subspace $\cH_n$
(the so-called finite-section) and, then, analyze the behavior of the eigenvalues of the matrix
$P_nT\upharpoonright \cH_n$ in the limit of large $n$. This method requires also 
additional conditions in order to guarantee a good approximation behavior of spectral objects.

The following classical approximation result for
Toeplitz operators due to Szeg\"o gives an example where the finite section
method can be used to approximate spectral measures. The following result can be seen
as a distributive version of Szeg\"o's classical limit theorems that involve determinants 
of Toeplitz matrices: 
denote by $\T$ the unit circle with normalized Haar measure $d\theta$
and consider the real-valued functions $g$ in $L^\infty(\T)$
which can be thought as (selfadjoint) multiplication operators on the 
complex Hilbert space $\cH:=L^2(\T)$, 
i.e., $M_g\,\varphi=g\;\varphi$, $\varphi\in\cH$. Denote by $P_n$ the finite-rank
orthogonal projection onto the linear span of $\{z^l\mid z\in\T, l=0,\dots,n\}$
and let $M_g^{(n)}$ be the corresponding finite section matrix. Write the corresponding
eigenvalues (repeated according to multiplicity) as $\{\lambda_{0,n},\dots,\lambda_{n,n}\}$.
Then, for any continuous $f\colon\R\to\R$ one has
\begin{equation}\label{szego}
 \lim_{n\to\infty}
  \frac{1}{n+1}
  \Big( f(\lambda_{0,n})+\dots+f(\lambda_{n,n})\Big) =\int_\T f(g(\theta)) \, d\theta
\end{equation}
(see \cite[Section~8]{Szego20}, \cite[Chapter~5]{bGrenander84} and \cite{WidomIn65}
for a careful analysis of this result; a recent standard book analyzing many aspects
of Toeplitz operators and containing a large number of references is 
\cite{bBoettcher06}). 
The equation (\ref{szego}) may be also reformulated in terms of weak*-convergence of 
the corresponding spectral measures and it allows to
approximate numerically the spectrum of $M_g$ 
in terms of the eigenvalues its finite sections (see \cite{ArvesonIn94}
as well as Chapter~7 in \cite{Roch08} and references cited therein).

Szeg\"o's classical result suggests the following question: what is the 
reason that guarantees the convergence of spectral measures and that 
can be possibly useful beyond the context of Toeplitz operators?
In the last two decades there has been a considerable application of
methods from operator algebras (mainly C*-algebras and von Neumann 
algebras\footnote{For the purposes of this article we will 
define a C*-algebra to be a *-subalgebra of $\cL(\cH)$ (the set 
of bounded linear operators in $\cH$) which is closed
in the topology defined by the operator norm $\|\cdot\|$. If $\cT\subset\cL(\cH)$
we will denote by $C^*(\cT)$ the C*-algebra generated by $\cT$. A von Neumann algebra
is an important subclass of the class of C*-algebras which is closed in the weak operator topology.})
to this question. In Subsection~\ref{subsec:foelner-algebra}
we will mention the F\o lner algebra, a unital C*-algebra 
which can be naturally defined for the finite section method. Moreover,
there are other natural C*-algebras that can be associated to sequences 
of finite sections (see Section~2 in \cite{Roch07} and references cited therein).

Arveson's seminal series of articles
\cite{Arveson93,Arveson94,ArvesonIn94} on these topics were directly inspired by 
Szeg\"o's theorem (see, e.g., \cite{Boettcher00,Kuijlaars00} for 
related developments in numerical analysis).
Among other interesting results, Arveson
gave conditions that guarantee that the essential spectrum of a
large class of selfadjoint
operators $T$ may be recovered from the sequence of eigenvalues
of certain finite dimensional compressions $T_n$
(see also Subsection~\ref{subsec:jacobi}). These results were then
refined by B\'edos who systematically applied the concept of F\o lner sequence 
to spectral approximation problems (see \cite{Bedos94,Bedos95,Bedos97}
as well as \cite{bHagen01}). 
It is stated in Section~7.2 of \cite{bHagen01} that SeLegue also considered
Szeg\"o-type theorems for Toeplitz operators in the context of C*-algebras.
Hansen extends some of the mentioned results
to the case of unbounded operators (cf.~\cite[\S~7]{Hansen08}; see also
\cite{Hansen11} for recent developments in the non-selfadjoint case). 
Brown shows in \cite{Brown06b} that abstract
results in C*-algebra theory can be applied to compute spectra of important
operators in mathematical physics like 
almost Mathieu operators or periodic magnetic Schr\"odinger operators on graphs.

We recall next two related notions that are important when addressing 
spectral approximation problems: 
let $\mathcal{T}\subset\cL(\cH)$ be a 
set of bounded linear operators 
on the complex separable Hilbert space $\mathcal{H}$. A sequence of non-zero  
finite rank orthogonal projections $\{P_i\}_{i\in \N}$ 
is called a F\o lner sequence for $\mathcal{T}$, if 
\begin{equation}\label{eq:foelner}
\lim_{i\to\infty} \frac{\|T P_i-P_i T\|_2}{\|P_i\|_2} = 0,\quad
T\in\cT\;,
\end{equation}
where $\|\cdot\|_2$ denotes the Hilbert-Schmidt norm. We call $\{P_i\}_{i\in \N}$
a proper F\o lner sequence if, in addition to (\ref{eq:foelner}), it is increasing
and converges strongly to $\1$.

The existence of a proper F\o lner sequence for a set of operators $\mathcal{T}$ is a
weaker notion than quasidiagonality which was introduced by
Halmos in the late sixties (cf.~\cite{Halmos70}).
Recall that a 
set of operators $\mathcal{T}\subset\mathcal{L}(\mathcal{H})$ 
is said to be quasidiagonal if there exists an increasing sequence of finite-rank projections
$\{P_i\}_{i}$ converging strongly to $\1$ and such that
\[
\lim_{i\to\infty}\|T P_i-P_i T\|=0\;,\quad\; T\in\mathcal{T}\;.
\]
It is easy to show that if $\{P_i\}_{i\in\N}$ quasidiagonalizes the set of operators
$\cT$, then it is also a proper F\o lner sequence for $\cT$ (for details see the next 
section). Moreover, Equation (\ref{eq:foelner}) can be
understood as a quasidiagonality condition, but relative to the growth of 
the dimension of the corresponding subspaces. 
Quasidiagonality is an important property in the analysis of the structure of 
C*-algebras (see, e.g., Chapter~7 in \cite{bBrown08} or 
\cite{BrownIn04,BlackadarIn04,Brown06a,Voiculescu93}) and is
also a very useful notion in spectral approximation problems. 
E.g.~quasidiagonality is assumed to prove the converge of 
spectra (in the selfadjoint case) and pseudospectra 
(cf.~\cite{Brown06} and \cite[\S~2]{Hansen08}). 
F\o lner sequences were introduced
in the context of operator algebras by Alain Connes in his seminal
article \cite[Section~V]{Connes76} (see also \cite{ConnesIn76,Popa86,PopaIn88}).
This notion is an algebraic analogue of F\o lner's characterization 
of amenable discrete groups (see Section~\ref{sec:foelner} for 
precise definitions) and was used by Connes as an essential tool in the
classification of injective type~II$_1$ factors.
The notion of proper F\o lner sequence is also interesting in the analysis of non-normal 
operators on Hilbert spaces and in relation of so-called finite operators (cf.~\cite{pLledoYakubovich12}).
In Section~\ref{sec:foelner} we state the main properties of
F\o lner sequences for operators that will be needed later and mention the relation to
quasidiagonality.

The third expression in the title of this article refers to crossed
products. This is a basic operator algebraic construction which is
interesting in its own right (see Section~\ref{sec:crossed} for details). 
The crossed product may be seen as a new C*-algebra 
constructed from a given C*-algebra which carries an action
of a group $\Gamma$. Many important operators in 
mathematical physics with very interesting spectral properties 
can be identified as elements of certain crossed products
(see, e.g., \cite{bBoca01,Lenz99} and Section~\ref{sec:applications} below).
The question when a crossed product of a quasidiagonal C*-algebra 
by an amenable group is 
again quasidiagonal has been addressed several times in the past
(see, e.g., Section~11 in \cite{BrownIn04} for some partial 
answers). In Lemma~3.6 of \cite{Pimsner80} it is shown that 
if $\cA$ is a unital separable quasidiagonal
C*-algebra with almost periodic group action 
$\alpha\colon\Z\to\mathrm{Aut}\cA$, then the C*-crossed-product
$\cA\rtimes_\alpha\Z$ is again quasidiagonal. This result has been
extended recently by Orfanos to C*-crossed products $\cA\rtimes_\alpha\Gamma$,
where now $\Gamma$ is a discrete, countable, amenable,
residually finite group (cf.~\cite{Orfanos10b}). In this more general case
there is again a certain compatibility condition on the group action of $\Gamma$ on $\cA$ is
needed. This result has been applied to generalized Bunce-Deddens algebras in
\cite{Orfanos10a}.

The existence of F\o lner sequences may 
be established in abstract terms, but this gives in general 
no clues of what are the concrete matrix approximations.
The aim of the present article is to give conditions and construct explicitly
F\o lner sequences for the crossed product of a C*-algebra $\cA$ that has
a F\o lner sequence and a discrete countable amenable group $\Gamma$
(see Theorem~\ref{pro:R} for a precise statement). We will state a sufficient
condition for the group action on the C*-algebra $\cA$ relative to the choices of F\o lner sequences
for the discrete group $\Gamma$. Our results
partly extend those of B\'edos for group von Neumann algebras
and are related to the articles of Orfanos
mentioned above. In Section~\ref{sec:applications} we will illustrate our
results in two well-known examples: Theorem~\ref{pro:R} can be applied
to the rotation algebra (also known as non-commutative torus), since it
can be seen as a crossed product of $C(\T)$ by $\Z$. 
This algebra contains interesting examples from the spectral point of view,
like almost Mathieu operators or discrete Schr\"odinger operators with
magnetic potentials, that have been widely studied in the literature
(see, e.g., \cite{bBoca01} and references cited therein). 
The rotation algebra provides also a non-trivial example were
we can verify the compatibility condition stated in Theorem~\ref{pro:R}.
The second example refers to the C*-algebra generated by bounded 
Jacobi operators on $\ell^2(\N)$. Also in this case we can have an interpretation
in terms of crossed products and we will mention some well-known results
for this class of operators (cf.~\cite{Roch08}). 
We conclude the article pointing out in Section~\ref{sec:outlook}
that Theorem~\ref{pro:R} extends the class of generalized band-dominated operators
as considered, e.g., in \cite{Roe05,pRabinovich10}.

\section{F\o lner sequences and quasidiagonality}\label{sec:foelner}
The notion of F\o lner sequences for operators has its origins in 
group theory. Recall that a
discrete countable group $\Gamma$ is amenable if it has an invariant
mean, i.e., there is a continuous linear functional $\psi$
on $\ell^\infty(\Gamma)$ with norm one and such that
\[
  \psi(u_\gamma f)=\psi(f)\;,\quad \gamma\in\Gamma\;,\quad f\in \ell^\infty(\Gamma)\;,
\]
where $u$ is the left-regular representation on $\ell^2(\Gamma)$.
A F\o lner sequence for 
$\Gamma$ is a sequence of non-empty finite subsets $\Gamma_i\subset\Gamma$ 
that satisfy
\begin{equation}\label{eq:foelner-g}
\lim_{i} 
\frac{|(\gamma \Gamma_i)\triangle
\Gamma_i|}{|\Gamma_i|} =0\qquad\text{for all}\quad \gamma\in\Gamma \,,
\end{equation}
where $\triangle$ denotes the symmetric difference and $|\Gamma_i|$
is the cardinality of $\Gamma_i$. Then, $\Gamma$ has a F\o lner sequence
if and only if $\Gamma$ is amenable (cf.~Chapter~4 in \cite{bPaterson88}).
Some authors require, in addition to Eq.~(\ref{eq:foelner-g}), that the 
sequence is increasing and complete, i.e., $\Gamma_i\subset\Gamma_j$ if 
$i\leq j$ and $\Gamma=\cup_i \Gamma_i$. We will not need these additional
assumptions here.

The counterpart of the previous definition in the context of operators is
given as follows:
\begin{definition}\label{def:Foelner}
Let $\mathcal{T}\subset\cL(\cH)$ be a set of bounded linear operators on the complex separable
Hilbert space $\mathcal{H}$. 
A sequence of non-zero 
finite rank orthogonal projections $\{P_n\}_{n\in \N}\subset\cL(\cH)$ 
is called a F\o lner sequence for $\mathcal{T}$ if  
\begin{equation}\label{eq:F1}
\lim_{n} \frac{\|T P_n-P_n T\|_2}{\|P_n\|_2} = 0\;\;,\quad T\in\cT\;,
\end{equation}
where $\|\cdot\|_2$ denotes the Hilbert-Schmidt norm. 
If the F\o lner sequence $\{P_n\}_{n\in \N}$ satisfies, in addition, that it is increasing
and converges strongly to $\1$, then we say it is a proper F\o lner sequence.
\end{definition}

We will state next some immediate consequences of the definition that will
be used later on.
To simplify expressions in the rest of the article we introduce the
commutator of two operators: $[A,B]:=AB-BA$.

\begin{proposition}\label{total}
Let $\cT\subset\cL(\cH)$ be a set of operators and $\{P_n\}_{n\in \N}$
a sequence of non-zero finite rank orthogonal projections.
\begin{itemize}
\item[(i)] $\{P_n\}_{n\in \N}$ is a F\o lner sequence
for $\cT$ iff it is a F\o lner sequence for $C^*(\cT\,,\,\1)$ 
(the C*-algebra generated by $\cT$ and the identity $\1$).

\item[(ii)]
Let $\cT$ be a selfadjoint set (i.e., $\cT^*=\cT$). Then
$\{P_n\}_{n\in \N}$ is a F\o lner sequence for $\cT$ if one of
the four following equivalent conditions holds for all $T\in\cT$:
\begin{equation}\label{F1}
\lim_{n} \frac{\|T P_n-P_n T\|_p}{\|P_n\|_p} = 0,\qquad
p\in\{1,2\}
\end{equation}
or
\begin{equation}\label{F2}
\lim_{n} \frac{\|(I-P_n)T P_n\|_p}{\|P_n\|_p} = 0,\qquad
p\in\{1,2\}\;,
\end{equation}
where $\|\cdot\|_1$ and $\|\cdot\|_2$ are the trace-class and Hilbert-Schmidt
norms, respectively.
\end{itemize}
\end{proposition}
\begin{proof}
(i) We just have to show that if $\{P_n\}_{n\in\N}$ is a F\o lner sequence for $\cT$ and $\1$, then
it is a F\o lner sequence for $C^*(\cT\,,\,\1)$. For $R,T\in\cT$ the following elementary 
relations
\begin{eqnarray*}
 \|[R\,T, P_n] \|_2 & \leq&\|R\,[T, P_n] \|_2 + \|[R, P_n]\,T \|_2 
                    \;\;\leq\;\; 
                    \|R\|\,\|[T, P_n] \|_2 + \|[R, P_n]\|_2\,\|T \| \\[2mm]
 \|[T^*, P_n] \|_2 &=& \|[T, P_n]^* \|_2 \;\;=\;\;\|[T, P_n] \|_2 
\end{eqnarray*}
show that $\{P_n\}_{n\in\N}$ is a F\o lner sequence for the *-algebra $\widetilde{\cT}$ generated 
by $\cT$. Then it is a standard $\frac{\varepsilon}{2}$-argument to show
that $\{P_n\}_{n\in\N}$ is still a F\o lner sequence for the norm closure of $\widetilde{\cT}$.

(ii) By the previous item we have that $\{P_n\}_{n\in \N}$ is a F\o lner sequence for
$\cT$ iff it is a F\o lner sequence for $C^*(\cT\,,\,\1)$ and
we can apply Lemma~1 in \cite{Bedos97}.
\end{proof}

The existence of a proper F\o lner sequence for an operator $T\in\cL(\cH)$ has the following absorbing 
property for direct sums. The proof of this fact is from the first versions of \cite{pLledoYakubovich12}.

\begin{proposition}\label{absorbing}
Let $\cH$ and $\cH'$ be separable Hilbert spaces with $\dim\cH=\infty$.
If $T$ has a proper F\o lner sequence, then $T\oplus X\in\cL(\cH\oplus\cH')$ has a 
proper F\o lner sequence for {\em any} $X\in\cL(\cH')$.
\end{proposition}
\begin{proof}
Let $\{P_n\}_{n\in \N}$ be a  proper F\o lner sequence for $T$ and since 
the sequence of projections is increasing we may assume 
that $\dim P_n\cH\geq n^2\;$. Let $\{\e_i\mid i\in\N\}$ be an orthonormal
basis of $\cH'$ and denote by $Q_n$ the orthogonal projection onto
span$\{e_1,\dots,e_n\}\subset\cH'$.
Then the following calculation shows that $\{P_n\oplus Q_n\}_n$ is a proper F\o lner sequence
for $T\oplus X$, $X\in\cL(\cH')$:
\begin{eqnarray*}
\frac{ \left\| \big[ T\oplus X,P_n\oplus Q_n\big]\right\|_2^2}{\|P_n\oplus Q_n\|^2_2}
  &=& \frac{\left\|[T,P_n]\right\|^2_2+\left\|[X,Q_n]\right\|_2^2}{\|P_n\|^2_2+n}\\
  &\leq &  \frac{\left\|[T,P_n]\right\|^2_2}{\|P_n\|^2_2}
          +\frac{4\left\|Q_n\right\|^2_2\;\left\|X\right\|^2}{n^2+n} \\
  &=& \frac{\left\|[T,P_n]\right\|^2_2}{\|P_n\|^2_2}
          +4\left\|X\right\| \frac{n}{n^2+n} \quad\to\quad 0\, 
\end{eqnarray*}
and the proof is concluded.
\end{proof}

The existence of a F\o lner sequence has important structural consequences. 
For the next result we need to recall the following notion:
a hypertrace for a C*-algebra $\cA$ acting on a Hilbert
space $\cH$ is a state $\Psi$ on $B(\cH)$ that is centralized by
$\cA$, i.e.
\begin{equation*}
\Psi(X A) = \Psi(A X)\;,\quad X\in B(\cH)\;,\;A\in\cA\,.
\end{equation*}
Hypertraces are the algebraic analogue of the invariant mean mentioned
at the beginning of this section (cf.~\cite{Connes76,ConnesIn76,Bedos97}).
We mention some easy operator algebraic consequence of the existence of a
F\o lner sequence for a C*-algebra (see also \cite{Bedos95,pAraLledop12}).
\begin{proposition} 
  Let $\cA\subset\cL(\cH)$ be a separable C*-algebra. If $\cA$ has a 
  F\o lner sequence, then $\cA$ has a hypertrace.
\end{proposition}

\begin{examples}\label{ex:shift}
\begin{itemize}
 \item[(i)] 
The unilateral shift is a canonical example that shows the difference between
the notions of F\o lner sequences and quasidiagonality. On the one hand,
it is a well-known fact that the unilateral unilateral shift $S$
is not a quasidiagonal operator. (This was shown by Halmos in \cite{Halmos68}; in 
fact, in this reference it is shown that $S$ is not even quasitriangular.) 
If $\cA$ is a (unital) C*-algebra 
containing a proper (i.e., non-unitary) isometry, then it {\em not} quasidiagonal
(see, e.g.~\cite{BrownIn04,bBrown08}). 
Finally, it can be shown that certain weighted shifts are quasidiagonal (cf.~\cite{Smucker82}).

On the other hand, it is easy to find a canonical proper F\o lner sequence for $S$. In fact, define
$S$ on $\cH:=\ell^2(\N_0)$ by $Se_i:=e_{i+1}$, where 
$\{e_i\mid i=0,1,2,\dots\}$ is the canonical basis of $\cH$
and consider for any $n$ the orthogonal projections $P_n$ onto span$\{e_i\mid i=0,1,2,\dots, n\}$.
Then
\[
 \big\|[P_n,S]\big\|_2^2=\sum_{i=1}^\infty \Big\|[P_n,S]e_i\Big\|^2=\|e_{n+1}\|^2=1
\]
and 
\[
\frac{\big\|[P_n,S]\big\|_2}{\| P_n\|_2}=\frac{1}{\sqrt{n+1}}\;
                                         \mathop{\longrightarrow}\limits_{n\to\infty} \;0\;.
\]
A similar argument shows directly that $\{P_n\}_n$ is a proper F\o lner sequence for any power
$S^k$, $k\in\N$. By Proposition~\ref{total}~(i) it follows that $\{P_n\}_n$ is also a proper F\o lner
sequence for the Toeplitz algebra $C^*(S)$. 

\item[(ii)] Let $T\in\cL(\cH)$ be a selfadjoint operator. If a sequence 
of non-zero finite rank orthogonal operators $\{P_n\}_n$ satisfies
\[
 \sup_{n\in\N}\|(\1-P_n)TP_n\|_2<\infty\;,
\]
then Eq.~(\ref{F2}) implies that $\{P_n\}_n$ is clearly a F\o lner sequence for $T$. Concrete examples satisfying
the preceding condition are selfadjoint operators with a band-limited
matrix representation (see, e.g.,~\cite{Arveson94,ArvesonIn94}).
Band limited operators together with quasidiagonal operators are the
essential ingredients in the solution of
Herrero's approximation problem, i.e.~the characterization of the closure
of block diagonal operators with bounded blocks
(see Chapter~16 in \cite{bBrown08} for a comprehensive presentation).

\end{itemize}
\end{examples}

\subsection{F\o lner algebra}\label{subsec:foelner-algebra}
In the study of growth properties of C*-algebras (and motivated
by previous work done by Arveson and B\'edos)
Vaillant defined the following
natural unital C*-algebra (see Section~3 in \cite{Vaillant96}):
given an increasing sequence $\cP:=\{P_n\}_n\subset\cL(\cH)$ 
of orthogonal finite rank projections strongly converging to 
$\1$, consider the set of all bounded
linear operators in $\cH$ that have $\cP$ as a proper F\o lner sequence, i.e.,
\[
 \cF_{\cP}(\cH):=\left\{
                 X\in \cL(\cH) \mid
                 \lim_{n\to\infty} \frac{\|X P_n-P_n X\|_2}{\|P_n\|_2} = 0
                 \right\}\;.
\]
This unital C*-subalgebra of $\cL(\cH)$ (called F\o lner algebra by Hagen, Roch
and Silbermann in Section~7.2.1 of \cite{bHagen01}) has shown to be very 
useful in the analysis of the classical Szeg\"o limit theorems for 
Toeplitz operators and some generalizations of them
(see, e.g., Section~7.2 of \cite{bHagen01} and \cite{Boettcher05}). 

In general, $\cF_{\cP}$ will not be separable for the operator norm.
This can be easily seen from the absorbing properties of proper F\o lner 
sequences for direct sums as shown in the proof of Proposition~\ref{absorbing}. 
In fact, it is possible that a F\o lner algebra contains a full copy
of some $\cL(\cH)$. 
Finally, let us mention that subalgebras of the F\o lner algebra
are also interesting from an abstract operatoralgebraic point of view. We refer
to \cite{Bedos95,pAraLledop12} for further developments in this direction.

\section{F\o lner sequences for crossed products}\label{sec:crossed}

The crossed product may be seen as a new C*-algebra 
constructed from a given C*-algebra which carries an action
of a group $\Gamma$. This procedure 
goes back to the pioneering work of Murray and von Neumann on rings of 
operators. Algebraically, this construction has some similarities 
with the semi-direct product of groups.
Standard references which present the crossed product construction
with small variations are \cite[Chapter~VIII]{bDavidson96},
\cite[Chapter~4]{bSunder87}, \cite[Section~V.7]{Takesaki:1} or
\cite[Section~8.6 and Chapter~13]{Kadison:Ringrose:2}. Since all groups $\Gamma$
considered will be amenable (and countable) 
we will not distinguish between crossed products and reduced crossed products.

Throughout this section $\mathcal{A}$ denotes a concrete C*-algebra
acting on a complex separable Hilbert space $\mathcal H$. We shall assume
that $\alpha$ is an automorphic representation of a countable discrete
amenable group $\Gamma$ on $\mathcal A$, i.e.
\begin{equation*}
\alpha\colon \Gamma \to\mathrm{Aut}\,\cA\;.
\end{equation*}

The crossed product is a new C*-algebra constructed with the
previous ingredients and acting on the separable Hilbert space
\begin{equation}\label{eq:decomp}
 \cK:=\ell^2(\Gamma )\otimes\cH\cong\mathop{\oplus}\limits_{\gamma\in\Gamma }\cH_\gamma\;,
\end{equation}
where $\cH_\gamma\equiv\cH$ for all $\gamma\in\Gamma $. To make
this notion precise we introduce representations of $\al A$ and $\Gamma $
on $\al K$: for $\xi=(\xi_\gamma)_\gamma\in\al K$ we define
\begin{eqnarray}\label{abelianM}
\left(\pi(M)\xi\right)_\gamma &:=&\alpha^{-1}_\gamma(M)\;\xi_\gamma
\;,\quad M\in\cA, \\ \label{lrr}
\left(U(\gamma_0\right)\xi)_\gamma&:=&\xi_{\gamma_0^{-1}\gamma} \,.
\end{eqnarray}

The crossed product of $\cA$ by $\Gamma $ is the C*-algebra on $\cK$
generated by these operators, i.e.,
\begin{equation*}
  \cN=\cA\rtimes_\alpha\Gamma 
    := C^*\Big(
        \left\{\pi(M)\mid M\in\cA\right\}\cup
        \left\{ U(\gamma)\mid\gamma\in\Gamma \right\}
       \Big)
     \subset\al L(\cK)\,,
\end{equation*} 
where $C^*(\cdot)$ denotes the C*-algebra generated by its argument.
A characteristic relation for the crossed product is
\begin{equation}\label{main}
\pi\left(\alpha_\gamma(M)\right)=U(\gamma)\pi(M)U(\gamma)^{-1}\,,
\end{equation}
that is, $\pi$ is a covariant representation of the $C^*$-dynamical system
$(\cA, \Gamma ,\alpha)$. 

\begin{remark}\label{circulant}
   Later we will need the following 
   useful operator matrix characterization of the elements in the 
   crossed product: consider
   the identification $\cK\cong\mathop{\oplus}\limits_{\gamma\in\Gamma }\cH_\gamma$ with
   $\cH_\gamma\equiv\cH$, $\gamma\in\Gamma$. 
   Then, every $T\in\cL(\cK)$ can be written as a 
   matrix $(T_{\gamma'\gamma})_{\gamma',\gamma\in\Gamma }$ with entries
   $T_{\gamma'\gamma}\in\cL(\cH)$. Any element $N$ in the crossed
   product $\cN\subset\cL(\cK)$ has the form
\begin{equation}\label{eq:matrix-n}
   N_{\gamma'\gamma}=\alpha_\gamma^{-1}\left( A(\gamma'\gamma^{-1})\right)
                   \;,\quad \gamma'\,,\,\gamma\in\Gamma \,,
 \end{equation}
 for some mapping $A\colon\Gamma \to\cA\subset\cL(\cH)$. Roughly, this means
 that the ``diagonals'' of the operator matrices of elements in the crossed
 product $\cN$ are orbits of the group action on elements in $\cA$.
  
  For example, the matrix form of the product of generators 
  $N:=\pi(M)\cdot U(\gamma_0)$, $M\in\cA$, $\gamma_0\in\Gamma$ is given by
\begin{equation}\label{matrix-generators}
 N_{\gamma'\gamma}=\alpha_{\gamma'}^{-1}(M) \, \delta_{\gamma',\gamma_0\gamma}
                  =\alpha_{\gamma'}^{-1}\left(A(\gamma'\gamma^{-1})\right)\;,
 \mathrm{where}\;
 A(\widetilde{\gamma}):=\left\{\begin{array}{l}
                                        M\;\mathrm{if}\;\;\widetilde{\gamma}=\gamma_0\\
                                        0\;\;\;\mathrm{otherwise}
                                      \end{array}\right. . 
\end{equation}
 
This implies that any function $A\colon \Gamma\to \cA$ with finite support determines
by means of Eq.~(\ref{eq:matrix-n}) an element in the crossed product.
\end{remark}

\subsection{Construction of a canonical F\o lner sequence}\label{subsec:foelner-cross}

The aim of the present section is to give a canonical example of  
a F\o lner sequence for the crossed product C*-algebra $\cN=\cA\rtimes_\alpha\Gamma$
constructed above. Since $\cN\subset\cL(\cK)$ with
$\cK=\ell^2(\Gamma )\otimes\cH$, our sequence is canonical in the 
sense that it uses explicitly a F\o lner sequence for $\Gamma$ and a sequence of projections on $\cH$ 
(cf.~Theorem~\ref{pro:R}).
We will also assume that the unital C*-algebra $\cA\subset\mathcal{L}(\cH)$ is
separable and has a F\o lner sequence $\{Q_i\}_{i\in\N}$. 

We begin recalling some parts of Proposition~4 in \cite{Bedos97}:
\begin{proposition}\label{pro:group-part}
Assume that the group $\Gamma$ is countable and amenable and denote by 
$\{P_i\}_{i\in\N}$ the sequence of orthogonal finite-rank projections 
from $\ell^2(\Gamma)$ onto $\ell^2(\Gamma_i)$ associated to 
a F\o lner sequence $\{\Gamma_i\}_{i\in\N}$ for the group $\Gamma$
(cf.~Section~\ref{sec:foelner}).
Then $\{P_i\}_i$ is a F\o lner sequence for the group
C*-algebra
\[
 \cA_\Gamma:=C^*\{\overline{U}(\gamma)\mid \gamma\in\Gamma\}\subset \cL(\ell^2(\Gamma))\;,
\]
where $\overline{U}$ is the left regular representation of $\Gamma$ on $\ell^2(\Gamma)$.
\end{proposition}

\begin{remark}
\item[(i)] 
In Proposition~4 of \cite{Bedos97} the author proves a stronger result. He shows 
that the canonical F\o lner net $\{P_i\}_{i\in\N}$ for the algebra and associated to the F\o lner net
of the (not necessarily countable) amenable group is still a F\o lner net for the corresponding
group von Neumann algebra, i.e., for the weak operator closure of $\cA_\Gamma$ in
$\cL(\ell^2(\Gamma))$. In general, it is not true that a F\o lner sequence for a concrete
C*-algebra is also a F\o lner sequence for its weak closure.
\item[(ii)] 
Recall that the preceding proposition means that the sequence $\{P_i\}_i$ satisfies
the four equivalent conditions in Proposition~\ref{total}~(ii) for any
element in the group von Neumann algebra.
\end{remark}

\begin{theorem}\label{pro:R}
Let $\cA\subset\mathcal{L}(\cH)$ be a separable C*-algebra 
which has a F\o lner sequence $\{Q_i\}_{i\in\N}$. Consider the 
countable and amenable group $\Gamma$ 
and denote by $\{P_i\}_{i\in\N}$ the sequence of orthogonal finite-rank projections 
from $\ell^2(\Gamma)$ onto $\ell^2(\Gamma_i)$ associated to 
a F\o lner sequence $\{\Gamma_i\}_{i\in\N}$ for the group $\Gamma$.
Assume that there is an action of $\Gamma$ on $\cA$ that satisfies:
\begin{equation}\label{compatibility}
  \lim_{i} 
  \left(
   \mathop{\mathrm{max}}\limits_{\gamma\in\Gamma_i}
   \frac{\|\left[Q_i,\alpha^{-1}_\gamma(M)\right]\|_{2}}{\|Q_i\|_2}
  \right)=0\;, \quad\mathrm{for~all}\;\; M\in\cA\;.
\end{equation}
Then the sequence
$\{R_{i}\}_{i\in \N}$ with
\[
 R_{i}:=P_i \otimes Q_i
\]
is a F\o lner sequence for the crossed product $\cN=\cA\rtimes_\alpha\Gamma$, 
i.e.~the four equivalent conditions in Proposition~\ref{total}~(ii) are satisfied.
\end{theorem}
\begin{proof}
Step~1:
we consider the identification
$\cK\cong\mathop{\oplus}\limits_{\gamma\in\Gamma}\cH_\gamma$, $\cH_\gamma\equiv\cH$.
In this case any element $N$ in the crossed
product $\cN\subset\cL(\cK)$ can be seen as a matrix of the form
\begin{equation}\label{eq:Nmatrix}
N_{\gamma'\gamma}=\alpha_\gamma^{-1}\left( A(\gamma'\gamma^{-1})\right)
                   \;,\quad \gamma'\,,\,\gamma\in\Gamma \,,
\end{equation}
where $\gamma\mapsto A(\gamma)$ is a mapping from $\Gamma\to\cA$
(cf.~Remark~\ref{circulant}).
Moreover, defining the unitary map 
\[
 W\colon\ell^2(\Gamma) \otimes \cH \to\mathop{\oplus}\limits_{\gamma\in\Gamma}\cH_\gamma\;,
 \quad \xi\otimes\varphi \mapsto (\xi_\gamma\varphi)_{\gamma\in\Gamma}
\]
it is straightforward to compute the matrix form of the projections 
$R_{i}=P_i \otimes Q_i\in\mathcal{L}(\mathcal{K})$:
\[
 (\widehat{R_{i}})_{\gamma'\gamma}:= (WR_{i}W^*)_{\gamma'\gamma}
         =\left\{ 
          \begin{array}{c}
           Q_i\; \delta_{\gamma'\gamma}\;,\quad \gamma',\gamma\in\Gamma_i \\[2mm]
           0 \;\;,\qquad\quad\mbox{otherwise}\,.
          \end{array}
          \right.
\]
The commutator of $\widehat{R_{i}}$ with any $N\in\cN$ is
\[
 \left[\widehat{R_{i}}\,,\,N\right]_{\gamma'\gamma}
         =\left\{ 
          \begin{array}{l}
           \left[Q_i\,,\, N_{\gamma'\gamma}\right]\;, 
                 \quad \gamma',\gamma\in\Gamma_i \\[3mm]
           Q_i\, N_{\gamma'\gamma}
               \;\;,\;\;\quad\quad\gamma\notin\Gamma_i\;,\;\gamma'\in\Gamma_i\\[3mm]
           -N_{\gamma'\gamma}\,Q_i
               \;\;,\quad\quad\gamma\in\Gamma_i\;,\;\gamma'\notin\Gamma_i\\[3mm]
           0 \;\;,\qquad\qquad\qquad
           \gamma\notin\Gamma_i\;,\;\gamma'\notin\Gamma_i\,.
          \end{array}
          \right.
\]

Step~2: we will check first the F\o lner condition on the 
product of generating elements 
$\pi(M)\,U(\gamma_0)$, $\gamma_0\in\Gamma$,
$M\in\cA$ (cf.~Eqs.~(\ref{abelianM}) and (\ref{lrr})). The corresponding matrix elements
are given according to Eq.~(\ref{matrix-generators}) by
\[
 N_{\gamma'\gamma}=\alpha_{\gamma'}^{-1}(M) \, \delta_{\gamma',\gamma_0\gamma}\,.
\]
Evaluating the commutator with $\widehat{R_{i}}$ on the basis elements 
$\{e_l\, \ot f._\gamma\}_{l,\gamma}\;$,
with $\ot f._\gamma(\gamma'):=\delta_{\gamma\gamma'}\;$, we get
\begin{equation}\label{eq:commutator}
 \left[\widehat{R_{i}}\,,\,(\pi(M) U(\gamma_0))\right]\;e_l\, \ot f._\gamma
         =\left\{ 
          \begin{array}{l}
   \left[Q_i\,,\, \alpha_{\gamma_0\gamma}^{-1}(M)\right]\,e_l\,\ot f._{\gamma_0\gamma}\;, 
                 \quad \gamma\in (\gamma_0^{-1}\Gamma_i)\cap\Gamma_i \\[3mm]
           Q_i\,\alpha_{\gamma_0\gamma}^{-1}(M) \,e_l\,\ot f._{\gamma_0\gamma}
        \;\;,\;\quad\quad \gamma\in(\gamma_0^{-1}\Gamma_i)\setminus\Gamma_i\\[3mm]
           -   \alpha_{\gamma_0\gamma}^{-1}(M)\,Q_i\,e_l\,\ot f._{\gamma_0\gamma}
        \;\;,\;\;\quad \gamma\in\Gamma_i\setminus(\gamma_0^{-1}\Gamma_i)\\[3mm]
           0 \;\;,\;\;\qquad\qquad\qquad\qquad\quad
           \gamma\notin\Gamma_i\;,\;\gamma\notin\gamma_0^{-1}\Gamma_i\,.
          \end{array}
          \right.
\end{equation}
From this we obtain the following estimates in the Hilbert-Schmidt norm:
\begin{eqnarray*}
\lefteqn{
\left\|\left[\widehat{R_{i}}\,,\,\pi(M) U(\gamma_0)\right]\right\|_2^2}\\[2mm]
    &=& \sum_{l,\gamma}\left\|\left[\widehat{R_{i}}\,,\,(\pi(M) U(\gamma_0))\right]
         \;e_l\, \ot f._\gamma\right\|^2 \\[2mm]
 &\leq& \sum_{\gamma\in(\gamma_0^{-1}\Gamma_i)\cap\Gamma_i} 
                       \left\| \left[Q_i\,,\, \alpha_{\gamma_0\gamma}^{-1}(M)\right]\right\|_2^2
        +2|(\gamma_0^{-1}\Gamma_i)\triangle\Gamma_i|  \,\|M\|^2\,\|Q_i\|_2^2\\[2mm]
 &\leq& |\Gamma_i| \;\mathop{\mathrm{max}}\limits_{\gamma\in\Gamma_i} \;
                        \left\| \left[Q_i\,,\, \alpha_{\gamma}^{-1}(M)\right]\right\|_2^2
        +2|(\gamma_0^{-1}\Gamma_i)\triangle\Gamma_i|  \,\|M\|^2\,\|Q_i\|_2^2\;.
\end{eqnarray*}
Using now the hypothesis (\ref{compatibility}) as well as the amenability of $\Gamma$ via
Eq.~(\ref{eq:foelner-g}) we get finally
\begin{equation*}
\frac{\left\|\left[\widehat{R_{i}}\,,\,\pi(M) U(\gamma_0)\right]\right\|_2^2}{
               \left\|\widehat{R_{i}}\right\|_2^2}
     \leq \mathop{\mathrm{max}}\limits_{\gamma\in\Gamma_i}
             \frac{\left\| \left[Q_i\,,\, \alpha_{\gamma}^{-1}(M)\right]\right\|_2^2}{\left\|Q_i\right\|_2^2}
             +2\,\|M\|^2\,\frac{|(\gamma_0^{-1}\Gamma_i)\triangle\Gamma_i|}{|\Gamma_i|}\;,
\end{equation*}
and
\[
 \lim_{i}\;
\frac{\left\|\left[\widehat{R_{i}}\,,\,\pi(M) U(\gamma_0)\right]\right\|_2^2}{
      \left\|\widehat{R_{i}}\right\|_2^2}
 =0\;,\quad M\in\cA\;,\;\gamma_0\in\Gamma\;.
\]
Thus, we have shown the F\o lner condition for the sequence
$\{R_{i}\}_{i}$ on the product of generating elements of 
the crossed product. By Proposition~\ref{total}~(i) we have that
$\{R_{i}\}_{i}$ is also a F\o lner sequence for their C*-closure 
\[
\cN=\cA\rtimes_\alpha\Gamma
:= C^*\Big(\left\{U(\gamma_0), \pi(M)
                     \mid \gamma_0\in\Gamma\,, M\in\cA\right\}\Big)\,.
\]
and the proof is concluded.
\end{proof}

The preceding result extends Proposition~\ref{pro:group-part} (proved by B\'edos), since
in the special case where $\cH$ is one-dimensional and 
$\cA\cong\C\1$, the crossed product reduces to the group 
C*-algebra $\cA_\Gamma$.

\begin{remark}\label{rem:trivial-compatibility}
The compatibility condition~(\ref{compatibility}) in the 
choices of the two F\o lner sequences requires some comments:
\begin{itemize}
 \item[(i)] Note that the compatibility condition already implies that the sequence $\{Q_i\}_i$
       must be a F\o lner sequence for the C*-algebra $\cA$. In fact, this is one of the 
       assumptions in Theorem~\ref{pro:R} that $\cA$ has a F\o lner sequence.
 \item[(ii)] Eq.~(\ref{compatibility}) is trivially satisfied in some cases: If $\Gamma$
  is finite, then the compatibility condition is a consequence of Eq.~(\ref{eq:F1}) 
  in Definition~\ref{def:Foelner}.

Another example is given by the crossed product
$\ell^\infty(\Gamma)\rtimes_\alpha \Gamma$, 
where $\Gamma$ is a discrete amenable group, $\ell^\infty(\Gamma)$ is the 
von Neumann algebra acting on the Hilbert space $\ell^2(\Gamma)$ by
multiplication
and the action $\alpha$ of $\Gamma$ on $\ell^\infty(\Gamma)$ is given by 
left translation of the argument. If $\{\Gamma_i\}_i$ is a F\o lner sequence for $\Gamma$ 
and we denote by $\{P_i\}_i$ the sequence of finite rank orthogonal projections 
from $\ell^2(\Gamma)$ onto $\ell^2(\Gamma_i)$, then it is easy to check
\[
 [P_i,g]=0\;,\quad g\in\ell^\infty(\Gamma)\;.
\]
Therefore, we have
\[
 \mathop{\mathrm{sup}}\limits_{\gamma\in\Gamma}
   \|\left[P_i,\alpha^{-1}_\gamma(g)\right]\|_{2}=0\;,\quad
  g\in  \ell^\infty(\Gamma)
\]
and we may apply Theorem~\ref{pro:R} to this situation.
This particular example is essentially the context in which 
B\'edos studies crossed products in Section~3 of 
\cite{Bedos97}. In fact, in this very special context
one can trace back the existence of a F\o lner sequence for the crossed product
to the amenability of the discrete group
(see Proposition~14 in \cite{Bedos97}).
\end{itemize}
\end{remark}

The content of Theorem~\ref{main} can be easily stated in terms of the
F\o lner algebra introduced in Subsection~\ref{subsec:foelner-algebra}. Let 
$\cQ:=\{Q_n\}_n\subset\cL(\cH)$ be an increasing sequence of finite rank orthogonal 
projections converging strongly to $\1$ and let $\cA$ be a unital separable C*-algebra
contained in the F\o lner algebra $\cF_{\cQ}(\cH)$. Moreover, consider the canonical sequence 
of projections $\cP:=\{P_n\}_n\subset\cL(\ell^2(\Gamma))$ associated to a F\o lner sequence
$\{\Gamma_n\}_n$ of the amenable discrete group $\Gamma$ as in Theorem~\ref{main}. 
(Here we also assume that the sequence $\Gamma_n$ is increasing and that 
$\Gamma=\cup_n\Gamma_n$.) Then, if $\cA$ carries a
compatible group action $\alpha$ (in the sense of Eq.~(\ref{compatibility})) we have
that the corresponding crossed product is a C*-subalgebra of the F\o lner algebra
$\cF_{\cR}\left(\ell^2(\Gamma)\otimes\cH\right)$, 
where $\cR:=\{P_n\otimes Q_n\}_n\subset\cL(\ell^2(\Gamma)\otimes\cH)$, i.e., we have
\[
 \cA\rtimes_\alpha\Gamma\subset \cF_{\cR}\Big(\ell^2(\Gamma)\otimes\cH \Big)\;.
\]

\section{Applications and examples}\label{sec:applications}

We will apply next the results of the previous sections in two
different situations. The two examples considered, rotation algebras and
Jacobi operators, can be interpreted in the context of crossed products. Both
examples are well-known and have been widely studied in the literature.
The rotation algebra gives a non-trivial example, where one can verify explicitly
the compatibility condition of Eq.~(\ref{compatibility}).

\subsection{The rotation algebra}\label{RotAlg}
The rotation algebra $\cA_\theta$, $\theta\in\R$, 
is the (universal) C*-algebra generated by two
unitaries $U$ and $V$ that satisfy the equation
\[
 U\,V=e^{2\pi i\theta}\, V\,U\;.
\]
When $\theta$ is an integer, the algebra $\cA_\theta$ is isomorphic to
the commutative C*-algebra $C(\T^2)$ of continuous functions on the 2-torus.
For this reason, when, e.g., $\theta$ is irrational, $\cA_\theta$ is 
called a non-commutative torus. Moreover, $\cA_\theta$
has in this case a unique faithful tracial state $\tau$ which can be interpreted
as a non-commutative analogue of the Haar measure on $\T^2$.
(See \cite{bBoca01} for a thorough presentation and many results concerning spectral
approximation.)

The rotation algebra is one of the fundamental examples in the theory C*-algebras 
and has been extensively used in mathematical physics. 
Interesting examples from the spectral point of view,
like the almost Mathieu operators or discrete Schr\"odinger operators with
magnetic potentials (Harper operators), can be identified as elements of the rotation algebra
(cf.~\cite{SunadaIn94,bBoca01}). In fact, consider for example the following representation
of the generators $U,V$ on $\cH:=\ell^2(\Z)$:
\[
 (U\xi)_k:=\xi_{k-1}\quad \mathrm{and}\quad (V\xi)_k:=e^{2\pi i\theta k} \,\xi_k \;,
\]
where $\xi=(\xi_k)_k\in\cH$. One defines the almost Mathieu operator with real parameters
$\theta,\lambda,\beta$ as
\begin{equation}\label{eq:mathieu}
 H_{\theta,\lambda,\beta}:=U+U^*+\frac{\lambda}{2}\left(e^{2\pi i\beta}V+ e^{-2\pi i\beta} V^*\right)\in \cA_\theta\,.
\end{equation}
These classes of operators have a natural generalization to arbitrary graphs.

An important fact for our purposes is that the rotation 
algebra can also be expressed as a crossed product
\[ 
\cA_\theta\cong C(\T)\rtimes_\alpha\Z\;\;, 
\]
where $C(\T)$ are the continuous
function on the unit circle and the action $\alpha\colon\Z\to C(\T)$ is given by 
rotation of the argument:
\begin{equation}\label{Rot}
 \alpha_k(f)(z):=f\left(e^{2\pi i k\theta}\,z\right) 
                  \;,\;\; f\in C(\T)\;,\;\; z\in\T\;.
\end{equation}

We will apply our main result to the C*-crossed product $\cA_\theta$.
Let $\{\epsilon_k(z):=z^k\mid k\in\Z\}$
be an orthonormal basis of Hilbert space $\cH:=L^2(\T)$ with the 
normalized Haar measure.
We choose (as in \cite[p.~216]{Bedos97}) 
a F\o lner sequence $\{Q_n\}_{n\in\N_0}$, where
$Q_n$ denotes the orthogonal projection onto the subspace generated by
$\{\epsilon_i\mid i=0,\dots n\}$. Moreover, we choose for the group 
$\Gamma=\Z$ the F\o lner sequence $\Gamma_n:=\{-n,-(n-1),\dots, (n-1), n\}$ and 
denote by $P_n$ the corresponding finite-rank orthogonal projections on
$\ell^2(\Z)$. First we verify that the compatibility condition~(\ref{compatibility}) 
for our choice of F\o lner sequences:
\begin{lemma}
 Consider the previous F\o lner sequence $\{Q_n\}_n$ for the commutative C*-algebra
 $\cA:=C(\T)$ and the group action $\alpha\colon\Z\to C(\T)$
 defined in Eq.~(\ref{Rot}). Then for $g\in C(\T)$ we have
\[
\left\|\left[Q_n,\alpha^{-1}_k(g)\right]\right\|_{2}=\|\left[Q_n, g\right]\|_{2}   \;,\quad k\in\Z\;,
\]
and
\[
  \lim_{n\to\infty} 
  \left(
   \mathop{\mathrm{max}}\limits_{k\in\Gamma_n}
   \frac{\|\left[Q_n,\alpha^{-1}_k(g)\right]\|_{2}}{\|Q_n\|_2}
  \right)=0\;, \quad\mathrm{for~all}\;\; g\in C(\T)\;.
\]

\end{lemma}\label{check:compatible}
\begin{proof}
The first equation is a consequence of the some elementary statements in harmonic analysis:
\begin{eqnarray*}
 \|\left[Q_n,\alpha^{-1}_k(g)\right]\|_{2}^2
  &=& \sum_{l=-\infty}^\infty 
      \|(Q_n\,\alpha_{-k}(g)-\alpha_{-k}(g)\, Q_n )\,\epsilon_l\|^2 \\[3mm]
  &=& \sum_{l=0}^n \|(\1-Q_n)\,\alpha_{-k}(g)\,\epsilon_l\|^2\;\;
      +\sum_{l\in(\Z\setminus\{0,\dots,n\})}\|Q_n\,\alpha_{-k}(g)\,\epsilon_l\|^2 \\[3mm]
  &=& \sum_{m\in(\Z\setminus\{0,\dots,n\}) }\;\; \sum_{l=0}^n 
      \left| e^{ 2\pi i\, k\theta (m-l)}\, \widehat{g}(m-l)\right|^2
      \\[3mm]
  & &+\sum_{m=0}^n \;\;\sum_{l\in(\Z\setminus\{0,\dots,n\})}
     \left| e^{2\pi i\, k\theta (m-l)}\, \widehat{g}(m-l)\right|^2
     \\[3mm]
  &=&\|\left[Q_n, g\right]\|_{2}^2\;.
\end{eqnarray*}
The second equation follows directly from the first equation and the fact that $\{Q_n\}_n$
is a F\o lner sequence for the algebra $C(\T)$.
\end{proof}

\begin{proposition}
 Let $\cA_\theta\cong C(\T)\rtimes_\alpha\Z$, with $\theta$ irrational, 
 be the C*-algebra associated to the rotation algebra and acting 
 on $\cK=\ell^2(\Z)\otimes\cH$.
\begin{itemize}
 \item[(i)] Consider the sequences $\{Q_n\}_{n\in\N_0}$ and $\{P_n\}_{n\in\N_0}$ defined
  before. Then 
  \[ 
\{R_n:=P_n\otimes Q_n\}_{n\in\N_0}
  \] 
 is a F\o lner sequence for $\cA_\theta$.
 \item[(ii)] Let $T\in\cA_\theta$ be a selfadjoint element in the rotation algebra and denote by
   $\mu_T$ the spectral measure associated with the unique trace of $\cA_\theta$. Consider the 
   compressions $T_n:=R_n T R_n$ and denote by $\mu_T^n$ the probability measures on $\R$ supported
   on the spectrum of $(T_n)$. Then $\mu_T^n\to \mu_T$ in the weak*-topology, i.e.
\[
 \lim_{n\to\infty}
  \frac{1}{d_n}
  \Big( f(\lambda_{1,n})+\dots+f(\lambda_{d_n,n})\Big) =\int f(\lambda) \, d\mu(\lambda)
  \;,\quad f\in C_0(\R)  \;,
\]
where $d_n$ is the dimension of the $R_n$ and $\{\lambda_{1,n},\dots,\lambda_{d_n,n}\}$
are the eigenvalues (repeated according to multiplicity) of $T_n$.
\end{itemize}
 
\end{proposition}
\begin{proof}
Part (i) follows from Theorem~\ref{pro:R} and  Lemma~\ref{check:compatible}. 
To prove Part (ii) recall $\cA_\theta$ 
has a unique trace (\cite{bBoca01}). 
The rest of the statement is a direct application of 
Theorem~6~(iii) in \cite{Bedos97} to the example of the rotation algebra.
\end{proof}

Since almost Mathieu operator $H_{\theta,\lambda,\beta}$ 
defined in Eq.~(\ref{eq:mathieu}) are selfadjoint
elements in $\cA_\theta$, we can apply
part (ii) of the precedent proposition. In this case the discrete measures $\mu_H^n$
are supported on the eigenvalues of the corresponding finite section matrices.


\subsection{Jacobi operators}\label{subsec:jacobi}
Jacobi operators have have been used in many branches of mathematics.
E.g., they can be interpreted as a discrete version of Schr\"odinger operators
and appear in the approximation of differential operators by difference operators
(see, e.g., \cite{Arveson93,lledo-post:08a}). Moreover,
the relation between selfadjoint tridiagonal infinite Jacobi matrices and orthogonal 
polynomials is by now a standard fact. 
In Chapter~2 of \cite{bDeift98} it is shown that there is a one-to-one correspondence
between bounded selfadjoint Jacobi operators $J$ and probability measures $\mu$ 
with compact support. The purpose of this subsection is to illustrate how naturally
the notion of a proper F\o lner sequence fits into the analysis of this class of operators.
The results are not new and we hope that we can make these techniques accessible to 
other communities presenting some elementary proofs of some of the statements.

Consider on $\cH:=\ell^2(\N_0)$ with canonical basis
$\{e_i\mid i=0,1,2,\dots\}$ the infinite Jacobi matrix
\begin{equation}\label{eq:J}
J=
\begin{pmatrix}
 a_0   &  b_0   &  0     & 0      & \dots \\
 b_0'  &  a_1   &  b_1   & 0      & \dots \\
 0     &  b_1'  &  a_2   & b_2    & \dots \\
\vdots & \vdots & \vdots & \vdots &
\end{pmatrix} \;.
\end{equation}
If the diagonals $\alpha=(a_0,a_1,\dots)$, $\beta=(b_0,b_1,\dots)$, $\beta'=(b_0',b_1',\dots)$
are bounded (i.e., $\alpha,\beta,\beta'\in\ell^\infty(\N_0)$), then $J$ is a bounded operator.
Moreover, if $D_\alpha=\mathrm{diag}(\alpha)$ denotes the diagonal operator, then $J$ can
be decomposed as
\begin{equation}\label{eq:J-decomp}
 J= D_\alpha + D_\beta\,S^* + S\,D_{\beta'}
   \;,
\end{equation}
where $Se_i:=e_{i+1}$ is the shift already considered in Example~\ref{ex:shift}~(i).

\begin{proposition}
Denote by $\cJ$ the set of all bounded Jacobi matrices as in Eq.~(\ref{eq:J}). Then
$\{P_n\}_n$, where $P_n$ are the orthogonal projections onto span$\{e_i\mid i=0,1,2,\dots, n\}$
is a proper F\o lner sequence for $C^*(\cJ)$.
\end{proposition}
\begin{proof}
By Proposition~\ref{total}~(i) it is enough to check Eq.~(\ref{eq:F1}) for the generating set $\cJ$.
For any $J\in\al J$ we have using the decomposition (\ref{eq:J-decomp})
\begin{eqnarray*}
\|[J,P_n]\|_2  &\leq& \underbrace{\|[D_\alpha,P_n]\|_2}_0 + \|[D_\beta S^*,P_n]\|_2 + \|[SD_{\beta'},P_n]\|_2 \\[2mm]
	       &\leq& \|D_\beta\| \; \|[S^*,P_n]\|_2 + \|[S,P_n]\|_2\|D_{\beta'}\|\;.
\end{eqnarray*}
F\o lner's condition in Eq.~(\ref{eq:F1}) follows from the computation in Example~\ref{ex:shift}~(i).
\end{proof}

Let $J\in\cJ$ be selfadjoint and denote by $\mu$ the spectral measure associated to the cyclic
vector $e_0$. The support of the (discrete) spectral measures $\mu_n^J$ of the finite sections
$J_n=P_nJP_n$ (with $P_n$ as in the preceding proposition)
correspond precisely with the zeros of the polynomials $p_n$ which are orthogonal with 
respect to $\mu$. For any Borel set $\Delta\subset\R$ define $N_n(\Delta)$ to be the number 
of eigenvalues of $J_n$ counting multiplicities contained in $\Delta$. Note that in this case we have
\[
 \mu_n^J(\Delta)=\frac{N_n(\Delta)}{n+1}\;.
\]
Following Arveson \cite{Arveson94,ArvesonIn94} we say that $\lambda\in\R$ is essential if for every
open set $\Delta\subset\R$ containing $\lambda$, we have 
\[
 \lim_{n\to\infty} N_n(\Delta) = \infty\;.
\]
The set of essential points is denoted by $\Lambda_{\mathrm{ess}}(J)$. Recall finally that in this context
the essential spectrum of the selfadjoint Jacobi operator $J$ is defined by
\[
 \sigma_{\mathrm{ess}}(J)=\sigma(J)\setminus\sigma_{\mathrm{disc}}(J)\;,
\]
where $\lambda\in\sigma_{\mathrm{disc}}(J)$ if it is an isolated eigenvalue in the spectrum
$\sigma(J)$ whose eigenspace is finite dimensional. Then since for tridiagonal Jacobi matrices 
we have that
\[
 \sup_n\mathrm{rank}\left(P_nJ-JP_n\right)\leq 2
\]
we can apply a theorem by Arveson to conclude that
\begin{equation}\label{eq:equalessential}
 \sigma_{\mathrm{ess}}(J)=\Lambda_{\mathrm{ess}}(J)\;.
\end{equation}
This result shows the way in which one can recover the essential spectrum of $J$ out of its
finite sections.

The results in this subsection have been extended far beyond the set of Jacobi operators.
In fact, the equality between the essential points and the essential spectrum has been
established for selfadjoint band-dominated operators (cf.~Theorem~7.6 in \cite{Roch08}).
We refer to Chapter~7 in \cite{Roch08} for details and a systematic analysis of many aspects of  
spectral approximation related to this class of operators.

\section{Outlook}\label{sec:outlook}

Jacobi operators are examples of so-called band dominated operators. This class of 
operators (already mentioned in the preceding section) can be identified with the 
crossed-product $\ell^\infty(\Z)\rtimes_\alpha\Z$, where the action $\alpha$ of 
$\Gamma$ on $\ell^\infty(\Z)$ is given by translation of the argument. This situation
can be generalized to the context of discrete groups where one just replaces 
$\Z$ by a discrete group $\Gamma$ (see \cite{Roe05,pRabinovich10} and references 
cited therein). In Theorem~2.2 of \cite{pRabinovich10} it is shown that the set 
of generalized band-dominated operators is precisely the reduced crossed product.
In particular, if $\Gamma$ is amenable (hence exact) we can apply also the result
of Theorem~\ref{main} to 	
$\ell^\infty(\Gamma)\rtimes_\alpha\Gamma$ (with $\alpha$ given again by left translation)
since by Remark~\ref{rem:trivial-compatibility}~(ii) the compatibility condition
is trivially satisfied.
From this point of view the result in Theorem~\ref{main} extends the set of 
generalized band-dominated operators by replacing the commutative algebra
$\ell^\infty(\Gamma)$ with a non-commutative C*-algebra $\cA$ which carries
a compatible group action of $\Gamma$ on $\cA$
in the sense of Eq.~(\ref{compatibility}). Moreover, our main result specifies a 
canonical F\o lner sequence where one could, in principle, address stability 
questions.
\vspace{1cm}


\paragraph{\bf Acknowledgements:} 
I would like to thank invitations of Vadim Kostrykin to the \textit{Universit\"at Mainz} 
(Germany) and of Pere Ara to the \textit{Universitat Aut\`onoma de Barcelona}. It is a 
pleasure to thank useful conversations with Pere Ara, Nate Brown, Marius Dadarlat,
Alfredo Dea\~no, Rachid el Harti, Vadim Kostrykin, Paulo Pinto and Dmitry Yakubovich 
on this subject. 


\end{document}